\documentclass[12pt,a4paper]{article}
\usepackage[utf8]{inputenc}
\usepackage[english]{babel}
\usepackage{hyperref}
\usepackage{amsmath}

\usepackage{multicol}
\usepackage{xcolor}
\usepackage {amsthm}
\theoremstyle{plain}
\newtheorem{thm}{Theorem}[section] \newtheorem{lem}[thm]{Lemma}  \newtheorem{cor}[thm]{Corollary}

\theoremstyle{definition}     

\theoremstyle{remark}  \newtheorem*{remark}{Remark}

\usepackage{upgreek}
\usepackage[shortlabels]{enumitem}
\usepackage[numbers]{natbib}

\usepackage{cancel}

\usepackage{amsfonts}
\usepackage{amssymb}
\usepackage{makeidx}
\usepackage{graphicx}
\usepackage{mathdots}
\usepackage{dsfont}
\allowdisplaybreaks

\usepackage{dsfont}

\begin{document}

\title{A Central Limit Theorem for the Ewens--Pitman random partition in the large-$\theta$ regime via a martingale approach}
\author{B. Bercu$^1$  \and C. Contardi$^2$ \and E. Dolera$^2$ \and S. Favaro$^3$ \and \\ {\small $^1$Institut de Math\'ematiques de Bordeaux, University of Bordeaux, France} \\ {\small $^2$Department of Mathematics, University of Pavia, Italy } \\ {\small $^3$Department of Economics and Statistics, University of Torino and Collegio Carlo Alberto, Italy}}

\date{}

\maketitle

\begin{abstract}
The Ewens-Pitman model defines a distribution on random partitions of $\{1,\ldots,n\}$, with parameters $\alpha \in [0,1)$ and $\theta > -\alpha$; the case $\alpha=0$ reduces to the classical Ewens model from population genetics. We investigate the large-$n$ asymptotic behaviour of the Ewens-Pitman random partition in the nonstandard regime $\theta=\lambda n$ with $\lambda>0$, establishing joint fluctuation results for the total number of blocks $K_n^{\{n\}}$ and the counts $K_{r,n}^{\{n\}}$ of blocks of sizes $r=1,\dots,d$, for fixed $d\in\mathbb{N}$. In particular, for $\alpha\in[0,1)$ and $\theta=\lambda n$, our main result provides a strong law of large numbers and a central limit theorem for the $(d+1)$-dimensional vector $\mathbf{K}_{d,n}^{\{n\}} = \bigl(K_n^{\{n\}}, K_{1,n}^{\{n\}}, \dots, K_{d,n}^{\{n\}}\bigr)^T$ as $n \to \infty$. The proof exploits the Chinese restaurant sequential construction under $\theta=\lambda n$ and a central limit theorem for triangular arrays of martingales, extending techniques previously developed for the classical regime with fixed $\theta$. As corollaries of our results, we recover known asymptotics for $K_n^{\{n\}}$ and derive new strong laws and central limit theorems for each fixed $K_{r,n}^{\{n\}}$, thereby completing earlier weak-law results and providing a comprehensive asymptotic description of the Ewens-Pitman partition structure in the large-$\theta$ setting.
\end{abstract}

\tableofcontents


\section{Introduction}
The Ewens-Pitman model for random partitions first appeared in \cite{Pit(95)} as a two-parameter generalization of the celebrated Ewens model in population genetics \cite{Ewe(72)}; see \cite{Cra(16)} and references therein. For $n\in\mathbb{N}$, denote by $\Pi_{n}$ a random partition of the set $\{1,\ldots,n\}$ into $K_{n}\in\{1,\ldots,n\}$ blocks of sizes (or empirical frequencies) $\mathbf{N}_{n}=(N_{1,n},\ldots,N_{K_{n},n})\in\mathbb{N}^{K_{n}}$ such that $n=\sum_{1\leq i\leq K_{n}}N_{i,n}$. For $\alpha\in[0,1)$ and $\theta>-\alpha$, the Ewens-Pitman model assigns to $\Pi_{n}$ the probability 
\begin{equation}\label{epsm}
P[K_{n}=k,\mathbf{N}_{n}=(n_{1},\ldots,n_{k})]=\frac{1}{k!}{n\choose n_{1},\ldots,n_{k}}\frac{(\theta)_{(k,\alpha)\uparrow}}{(\theta)_{n \uparrow}}\prod_{i=1}^{n}(1-\alpha)_{(n_{i}-1) \uparrow},
\end{equation}
where $(x)_{(n,a) \uparrow}$ denotes the rising factorial of $x$ of order $n$ and increment $a$, i.e. $(x)_{(n,a)\uparrow} :=\prod_{0\leq i\leq n-1}(x+ia)$ and $(x)_{n\uparrow} :=(x)_{(n,1)\uparrow }$. The distribution \eqref{epsm} admits a sequential construction in terms of the Chinese restaurant process \cite{Pit(95),Zab(05)} and a Poisson process construction by random sampling the two-parameter Poisson-Dirichlet distribution \cite{Per(92),Pit(97)}; see also \cite{Dol(20),Dol(21)} for a construction through the negative-Binomial compound Poisson model for random partitions \cite{Cha(07)}. For $\alpha=0$ the Ewens-Pitman model reduces to the Ewens model, arising by random sampling the Poisson-Dirichlet distribution \cite{Kin(75)}. The Ewens-Pitman model plays a critical role in a variety of research areas, e.g., population genetics, Bayesian statistics, combinatorics, machine learning and statistical physics. See \cite[Chapter 3]{Pit(06)} for an overview of the Ewens-Pitman model.

Under the Ewens-Pitman model \eqref{epsm}, there have been several works on the large $n$ asymptotic behaviour of $K_{n}$, showing different behaviours depending on whether $\alpha=0$ or $\alpha\in(0,1)$. Throughout, denote the almost sure and weak convergence by $\stackrel{a.s.}{\longrightarrow}$ and $\stackrel{w}{\longrightarrow}$, respectively. For $\alpha=0$ and $\theta>0$, \cite[Theorem 2.3]{Kor(73)} and Lindeberg-Feller central limit theorem show that as $n\rightarrow+\infty$
\begin{equation}\label{llnk_DP}
\frac{K_{n}}{\log n}\stackrel{a.s.}{\longrightarrow}\theta
\end{equation}
and 
\begin{equation}\label{cltk_DP}
\sqrt{\log n}\left(\frac{K_{n}}{\log n}-\theta\right)\stackrel{w}{\longrightarrow}\sqrt{\theta}\,\mathcal{N}(0,1),
\end{equation}
where $\mathcal{N}(0,1)$ denotes the standard Gaussian random variable. Instead, for $\alpha\in(0,1)$ and $\theta>-\alpha$, \cite[Theorem 3.8]{Pit(06)} and \cite[Theorem 2.1 and Theorem 2.3]{BF24} show that as $n\rightarrow+\infty$
\begin{equation}\label{as_limit}
\frac{K_{n}}{n^{\alpha}}\stackrel{a.s.}{\longrightarrow}S_{\alpha,\theta}
\end{equation}
and 
\begin{equation}\label{gaussian_limit_kn}
\sqrt{n^{\alpha}}\left(\frac{K_{n}}{n^{\alpha}}-S_{\alpha,\theta}\right)\stackrel{w}{\longrightarrow}\sqrt{\tilde{S}_{\alpha,\theta}}\,\mathcal{N}(0,1),
\end{equation}
where $S_{\alpha,\theta}$ and $\tilde{S}_{\alpha,\theta}$ are scaled Mittag-Leffler random variables \cite{Zol(86)}, sharing the same distribution, and $\tilde{S}_{\alpha,\theta}$ is independent of $\mathcal{N}(0,1)$. See \cite[Chapter 3]{Pit(06)} for details on \eqref{llnk_DP} and \eqref{as_limit}.

An equivalent description of the partition $\Pi_n$ is given, for every $n \in \mathbb{N}$, by the sequence $\mathbf{K}_{n} : = (K_n, K_{1,n}, ... , K_{n,n}, K_{n+1, n}, ...) $, where for $r \in \mathbb{N}$, $K_{r,n}$ denotes the number of blocks of $\Pi_n$ of size $r$, that is
\begin{displaymath}
K_{r,n} =\sum_{k = 1}^{K_n}\mathds{1}_{\{N_{k,n}=r\}}
\end{displaymath}
with $\mathds{1}$ denoting the indicator function. Clearly, for $r >n$, $K_{r,n} = 0$ a.s., and furthermore it is easy to verify that
$$
\sum_{r=1}^n K_{r,n} = K_n\quad  \text{and} \quad
\sum_{r=1}^n r \,  K_{r,n} = n.
$$
These quantities arise naturally from the sequential construction of the Ewens--Pitman distribution; see appendix \ref{app_pyp} for details. There are several known results regarding the asymptotic behaviour of $\mathbf{K}_n$ under the Ewens--Pitman model. In particular, \cite[Theorem 1]{Arr(92)} entails that for $\alpha = 0$
\begin{equation}\label{w_e_1}
K_{r,n}
\underset{n\rightarrow+\infty}{\overset{w}{\longrightarrow}}
Z_{r}
\end{equation}
for every $r \in \mathbb{N}$, 
and
\begin{equation}\label{w_e_2}
(K_{1,n}, K_{2,n}, \ldots)
\underset{n\rightarrow+\infty}{\stackrel{w}{\longrightarrow}} (Z_{1},Z_{2},\ldots),
\end{equation}
where the $Z_{r}$'s are independent Poisson random variables with $\mathbb{E}[Z_{r}]=\theta/r$, for all $r\geq1$. See \cite{Arr(03)} for generalizations and refinements of these asymptotic results. For $\alpha\in(0,1)$, it follows from \cite[Lemma 3.11]{Pit(06)}  that
\begin{equation}\label{as_ep_2}
\lim_{n\rightarrow+\infty}\frac{K_{r,n}}{n^{\alpha}}=p_{\alpha}(r)S_{\alpha,\theta}\qquad\text{a.s.},
\end{equation}
where
\begin{displaymath}
p_{\alpha}(r)=\frac{\alpha(1-\alpha)_{(r-1)\uparrow}}{r!} .
\end{displaymath}
We refer to  \cite[Theorem 3.3. and Theorem 3.4]{BF24} for Gaussian fluctuations and laws of iterated logarithm for $K_{r,n}$.  
%

\begin{remark} 
Beyond the almost-sure convergence and Gaussian fluctuations, $K_{n}$ and $K_{r, n}$ have been investigated with respect to large and moderate deviations \cite{Fen(98),Fav(14),Fav(18),BF25}. Non-asymptotic results for $K_{n}$ have been established in terms of Berry-Esseen theorems \cite{Dol(20)} and concentration inequalities \cite{Per(22),BF25}.
\end{remark}

\subsection{Main result}
Under the Ewens-Pitman model \eqref{epsm}, we study the large $n$ asymptotic behaviour of the random $(d+1)-$dimensional vector
\begin{displaymath}
\mathbf{K}_{d,n}^{\{n\}} = \left(K_n^{\{n\}},  K_{1, n}^{\{n\}}, ..., K_{d, n}^{\{n\}} \right)^T
\end{displaymath}
for fixed $d \in \mathbb{N}$, where the superscript $\{n\}$ denotes the fact that the parameter $\theta$ is allowed to depend linearly on $n\in\mathbb{N}$, namely $\theta=\lambda n$ with $\lambda>0$. This non-standard asymptotic regime first appeared in \cite{Fen(07)} for the the Ewens model, i.e. $\alpha=0$. In particular,  \cite[Proposition 2 and Theorem 2]{Tsu(17)} established a law of large numbers (LLN) and a central limit theorem (CLT) for  $K^{\{n\}}_n$ for $\alpha=0$ and $\theta=\lambda n$, later extended to $\alpha \in [0,1)$ by \cite[Theorem 1.2]{Con(25a)} and \cite[Theorems 1 and 2]{Rib(25)}. Further, \cite[Proposition 1]{Ber25} established a weak law of large number for the $K_{r,n}^{\{n\}}$. 

Denote by $\mathcal{S}_{+}^{d+1}(\mathbb{R})$ the set of symmetric and positive-semidefinite $(d+1) \times (d+1)$ real matrices. The next theorem states the main result of the paper, namely a LLN and a CLT for $\mathbf{K}_{d,n}^{\{n\}}$.

\begin{thm}\label{main_d}
Fix $d \ge 1$. For $n\in\mathbb{N}$, let $\mathbf{K}_{d,n}^{\{n\}}$
 be the random vector containing the number of partition blocks and the number of partition blocks of sizes $1, ..., d$ under the Ewens-Pitman model with parameters $\alpha\in[0,1)$ and $\theta=\lambda n$, with $\lambda>0$. Then, as $n\rightarrow+\infty$ there hold:
\begin{itemize}
\item[i)]
\begin{equation}\label{mom_m_d}
\mathbb{E}\left[\mathbf{K}_{d, n}^{\{n\}}\right] =n\,\mathfrak{M}_{d, \alpha, \lambda} + O(1)
\end{equation}
\item[ii)] (SLLN)
\begin{equation}\label{lln_d}
\frac{\mathbf{K}_{d, n}^{\{n\}}}{n}\stackrel{a.s.}{\longrightarrow}\mathfrak{M}_{d, \alpha, \lambda};
\end{equation}
\item[iii)] (CLT)
\begin{equation}\label{clt_d}
\frac{\mathbf{K}_{d, n}^{\{n\}}-n\, \mathfrak{M}_{d, \alpha, \lambda}}{\sqrt{n}}\stackrel{\text{w}}{\longrightarrow}\mathcal{N}(\mathbf{0},\Sigma_{d, \alpha, \lambda}).
\end{equation}
\end{itemize}
where $\mathbf{\mathfrak{M}}_{d, \alpha,\lambda} \in \mathbb{R}^{d+1}$
and $\Sigma_{d, \alpha, \lambda} \in \mathcal{S}_{+}^{d+1}(\mathbb{R})$ are explicit (see equations \eqref{mr} and \eqref{Sigma_d} below) and only depend on $\alpha$ and $\lambda$. 
\end{thm}

Theorem \ref{main_d} provides a novel, comprehensive description of the behaviour of the Ewens-Pitman random partition under the non-standard asymptotic regime $\theta=\lambda n$, with $\lambda>0$. By comparing the LLN \eqref{lln_d} with the almost-sure fluctuation results \eqref{llnk_DP}, \eqref{as_limit}, \eqref{w_e_2} and \eqref{as_ep_2}, we can assess how the linear dependence of $\theta$ on $n$, i.e. $\theta=\lambda n$, affects the large $n$ asymptotic behaviour of $K_{n}$ in terms of both scaling and limiting behaviour. Specifically, while the almost-sure fluctuations exhibit different scalings depending on the value of  $\alpha$, and for $K_n$ as compared to $K_{r,n}$, the LLN \eqref{lln_d} exhibits the ``usual" linear scaling in $n$ for all $\alpha\in[0,1)$ and for $K_n^{\{n\}}$ as well as all the $K_{r,n}^{\{n\}}$. Moreover, in the almost-sure fluctuation results \eqref{llnk_DP}, \eqref{as_limit} and \eqref{as_ep_2} the limiting behaviour is non-random for $\alpha=0$ and random for $\alpha\in(0,1)$, which in turn determines the non-random and random centerings in \eqref{cltk_DP} and \eqref{gaussian_limit_kn} , respectively. In contrast, the LLN \eqref{lln_d} yields a non-random limit for all $\alpha\in[0,1)$,  leading to a non-random centering in the CLT \eqref{clt_d}.

The proof of Theorem \ref{main_d} relies on the sequential construction of the Ewens--Pitman model in  the regime $\theta = \lambda n$. The proof of the CLT \eqref{clt_d} follows by adapting a martingale argument used in \cite{BF24} to the setting of such construction. However, in our non-standard large-$\theta$ regime, it is necessary to make use of a more sophisticated central limit theorem for triangular arrays of martingales.

The explicit expressions for $\mathbf{\mathfrak{M}}_{d, \alpha,\lambda}$ and $\Sigma_{d, \alpha, \lambda}$ are as follows. As regard $\mathbf{\mathfrak{M}}_{d, \alpha,\lambda}$, it is defined as
\begin{equation} \label{mr_vec}
\mathfrak{M}_{d, \alpha, \lambda} = \left(\mathfrak{m}_{0; \alpha, \lambda}, \mathfrak{m}_{1; \alpha, \lambda}, ..., \mathfrak{m}_{d; \alpha, \lambda}\right)^T
\end{equation}
where
\begin{equation} \label{mr}
\mathfrak{m}_{r; \alpha, \lambda} = \begin{cases}\frac{\lambda}{\alpha} \left[ \left(\frac{\lambda +1}{\lambda}\right)^\alpha - 1 \right] &  \text{if } r = 0, \, \alpha \in (0, 1)\\[0.4cm]
\lambda \log \left(\frac{\lambda +1}{\lambda}\right) & \text{if } r=0, \, \alpha = 0\\[0.4cm]
\frac{ (1-\alpha )_{\uparrow r-1}}{r! } \,  \lambda^{1-\alpha} (1+\lambda)^{\alpha -r} & \text{if } r\ge 1, \, \alpha \in [0,1).
\end{cases}.
\end{equation}
As regard $\Sigma_{d, \alpha, \lambda}$, it is defined as 
\begin{equation}\label{Sigma_d}
\left(\Sigma_{d, \alpha, \lambda}\right)_{i,j} = \left(\frac{\lambda+1}{\lambda}\right)^{i+j-2-2\alpha}
\left(\Gamma_{d, \alpha, \lambda}\right)_{i,j}
\end{equation}
where $\Gamma_{d, \alpha, \lambda}$ is a symmetric matrix defined by
\begin{align}\label{Gamma_d} & \left(\Gamma_{d, \alpha, \lambda}\right)_{i,j} = \begin{cases}
s^2_{i-1; \alpha, \lambda} & i = j \\[0.4cm]
\frac{\lambda}{\alpha -1}  - \frac{(\lambda + 1)}{\alpha -1} \cdot  \left( \frac{\lambda}{\lambda+1} \right)^{\alpha}- \left(\frac{\lambda}{\lambda+1}\right)^2+ \frac{1-\alpha }{2\,  (\lambda+1)}\cdot H(1,2) & i = 1, j = 2\\[0.4cm]
\lambda^{2-j} \left[
- c_{j-2} \frac{(\lambda+1)^{-3}}{j-1} \cdot H (j-3; j-1) + c_{j-1} \frac{(\lambda+1)^{-2}}{j}  \cdot H(j-1; j)
\right] & i = 1, j \ge 3\\[0.4cm]
-  \lambda^{1+\alpha-2i} c_{i-1} \frac{(\lambda+1)^{i-1-\alpha}}{i}  \cdot H( i-\alpha; i) - \lambda^{2-2i} \cdot \varrho_{\alpha, \lambda}(i-2, i-1)
& i \ge 2, j = i+1\\[0.4cm]
- \lambda^{3-i-j}  \cdot \varrho_{\alpha, \lambda}(i-2, j-2)
& i \ge 2, j \ge i+2
%
\end{cases}
\end{align}
with $H (b, c)  = \, _2F_1 \left(1, b ; c \, ; -1\lambda\right)$
where $_2F_1(a, b; c; z)$ stands for the Gauss hypergeometric function defined, for all $z \in \mathbb{C}$ where $|z|<1$, by 
\begin{equation*} \label{Gausshyper}
_2F_1(a, b; c; z)=\, \sum_{r=0}^\infty \frac{(a)_{r \uparrow} (b)_{r \uparrow}} {(c)_{r \uparrow}}\frac{z^r}{r!}; 
\end{equation*}
see \cite[Equation \href{https://dlmf.nist.gov/15.2.E1} {(15.2.1)}]{nist}. Further, for $r,s \in \mathbb{N}$, 
$$
c_r = \frac{(1-\alpha)_{r \uparrow}}{r!} = \frac{r}{\alpha} \cdot p_\alpha(r)
$$
and 
\begin{align*}
\varrho_{\alpha, \lambda}(r,s) & = c_r c_s\,  \frac{(\lambda+1)^{-4}}{r+s+1} H(r+s+2, r+s+1) \\
& \quad - (c_{r+1} c_s + c_r c_{s+1})\,  \frac{(\lambda+1)^{-3}}{r+s+2} H(r+s+3, r+s+2)\\
& \quad +c_{r+1} c_{s+1}\,  \frac{(\lambda+1)^{-2}}{r+s+3} H(r+s+4, r+s+3).
\end{align*}
Finally, 
\begin{equation}\label{s0}
s^2_{0; \alpha, \lambda} = \begin{cases} \frac{\lambda}{\alpha} \cdot \left[ \frac{\lambda +1-\alpha}{\lambda +1} - \left(\frac{\lambda}{\lambda +1}\right)^\alpha\right]  & \text{if } \alpha \in (0,1) \\
\log\left(\frac{\lambda+1}{\lambda}\right)  - \frac{1}{\lambda +1}& \text{if } \alpha = 0, 
 \end{cases}
 \end{equation}
 and  for $i \ge 2$, 
\begin{align} \label{sr}
\nonumber
s_{i-1; \alpha, \lambda}^2 &= \lambda^{\alpha + 2 -2i} \left[
c_{i-2} \frac{(\lambda+1)^{i-3-\alpha}}{i-1} \cdot  H( i-\alpha-2; i-1)\right. \\
& \quad   \left. + c_{i-1} \frac{(\lambda+1)^{i-2-\alpha}}{i}  \cdot H( i-\alpha; i)
\right]  - \lambda^{3-2i} \cdot \varrho_{\alpha, \lambda}(i-2, i-2).
\end{align}


As corollaries of Theorem \ref{main_d}, we obtain results for the marginal distributions of $K_n^{\{n\}}$ and of the $K_{r,n}^{\{n\}}$. For what concerns $K_{n}^{\{n\}}$, we recover \cite[Theorem 1.2]{Con(25a)}, hence providing an alternative proof. 

\begin{cor}[Theorem 1.2 of \cite{Con(25a)}]
For $n\in\mathbb{N}$, let $K_{n}^{\{n\}}$ be the number of partition blocks under the Ewens-Pitman model with parameter $\alpha\in[0,1)$ and $\theta=\lambda n$, with $\lambda>0$. 
as $n\rightarrow+\infty$ there hold:
\begin{itemize}
\item[i)]
\begin{equation*}\label{mom_m}
\mathbb{E}\left[K_n^{\{n\}}\right] =n\, \mathfrak{m}_{0;\alpha, \lambda} + O(1)
\end{equation*}
\item[ii)]
\begin{equation}\label{lln_kn_old}
\frac{K_{n}^{\{n\}}}{n}\stackrel{a.s.}{\longrightarrow}\mathfrak{m}_{0;\alpha, \lambda};
\end{equation}
\item[iii)]
\begin{equation}\label{clt_kn_old}
\frac{K_{n}^{\{n\}}-n\mathfrak{m}_{0;\alpha, \lambda}}{\sqrt{n s_{0;\alpha, \lambda}^{2}}}\stackrel{\text{w}}{\longrightarrow}\mathcal{N}\left(0,\Big(\frac{\lambda}{\lambda+1}\Big)^{2\alpha}\right).
\end{equation}
\end{itemize}
\end{cor}

The result regarding the marginal distributions of the $K_{r,n}^{\{n\}}$ is new, and it completes the analysis begun in \cite[Proposition 1]{Ber25}, strenghtening the weak LLN to a strong LLN and establishing a CLT.
\begin{cor}
For $r \le n\in\mathbb{N}$, let $K_{r,n}^{\{n\}}$ denote the number of partition blocks of size $r$ under the Ewens-Pitman model with parameter $\alpha\in[0,1)$ and $\theta=\lambda n$, with $\lambda>0$. 
as $n\rightarrow+\infty$ there hold:
\begin{itemize}
\item[i)]
\begin{equation*}\label{mom_m_r}
\mathbb{E}\left[K_{r,n}^{\{n\}}\right] =n\, \mathfrak{m}_{r; \alpha, \lambda} + O(1)
\end{equation*}
\item[ii)]
\begin{equation}\label{lln_krn}
\frac{K_{r, n}^{\{n\}}}{n}\stackrel{a.s.}{\longrightarrow}\mathfrak{m}_{r; \alpha, \lambda};
\end{equation}
\item[iii)]
\begin{equation}\label{clt_krn}
\frac{K_{n}^{\{n\}}-n\mathfrak{m}_{r; \alpha, \lambda}}{\sqrt{n s_{r; \alpha, \lambda}^{2}}}\stackrel{\text{w}}{\longrightarrow}\mathcal{N}\left(0,\Big(\frac{\lambda}{\lambda+1}\Big)^{2\alpha-r}\right).
\end{equation}
\end{itemize}
where $ \mathfrak{m}_{r; \alpha, 	\lambda} $ is as in \eqref{mr}. 
\end{cor}

\subsection{Related work}
There exists a rich literature on the large-$\theta$ asymptotic behaviour of the two-parameter Poisson-Dirichlet distribution, assuming $\alpha\in[0,1)$ and $\theta>0$. The genetic interpretation of $\theta$ has motivated the study of the large-$\theta$ asymptotic behaviour of the Poisson-Dirichlet distribution, as well as of statistics thereof, providing Gaussian fluctuations and large (and moderate) deviations \cite{Wat(77),Gri(79),Joy(02),Daw(06),Fen(07),Fen(08)}. As already mentioned, \cite{Fen(07)} first considered the regime $\theta=\lambda n$, with $\lambda>0$, in the study of the large $n$ asymptotic behaviour of the number $K_{n}^{\{n\}}$ of blocks in the Ewens model, providing a large deviation principle for $K_{n}^{\{n\}}$. Subsequently, \cite{Tsu(17)} derived both a LLN and a CLT for $K_{n}^{\{n\}}$ under the same regime. The large-$\theta$ asymptotic behaviour of $K_{r,n}^{\{n\}}$ has been explored in \cite{Ber25}, establishing a microclustering property, namely the size of the largest cluster grows sub-linearly with the sample size, while the number of clusters grows linearly, by means of a weak law of large numbers.  Some of the large-$\theta$ asymptotic results for the Poisson-Dirichlet distribution have been extended to two-parameter Poisson-Dirichlet distribution (i.e. $\alpha\in(0,1)$), in terms of both Gaussian fluctuations and large deviations \cite{Fen(07a),Fen(10)}. Further, the already mentioned \cite[Theorem 1.2]{Con(25a)} extended the results of \cite{Tsu(17)} to the general case $\alpha \in [0,1)$. The large-$\theta$ asymptotic behaviour of $K_{r,n}^{\{n\}}$ has been explored in \cite{Ber25} along with the case $\alpha = 0$, establishing a weak version of \eqref{lln_d}.


\subsection{Organization of the paper}
Section \ref{s2} contains the proof of Theorem \ref{main_d}, with technical results deferred to appendices \ref{app3_1}, \ref{sec_martingale} and \ref{app3_2}.

\section{Proof of Theorem \ref{main_d}}\label{s2}
To prove Theorem \ref{main_d}, we adapt the martingale argument first used in \cite{BF25} to the large-$\theta$ setting. 
Our argument relies heavily on the sequential construction of the Ewens--Pitman random partition under the $\theta = \lambda n$ hypothesis, which we introduce in subsection \ref{CRP_arrays_sec}. Following \cite{BF25},  throughout this section we use the convention 
$$K_{0, h}^{\{n\}} : = K_{h}^{\{n\}}. $$

\subsection{Sequential construction of the partition in the large-$\theta$ regime}\label{CRP_arrays_sec}

It is well-known that the (standard) Ewens--Pitman model admits a sequential construction in terms of the Chinese Restaurant Process (CRP) -- see Appendix \ref{app_pyp} for details. The model in the ``large-$\theta$'' setting also admits a sequential construction in terms of the CRP. However, this requires a careful specification of the objects at hand, to avoid confusion with the model in the classical setting. In particular, the sequential construction in this case is used to build finite arrays of partitions, defined as follows.\\

For $n \in \mathbb{N}$ and $h \in \{1, ..., n\}$, let 
$$\mathbf{K}_{h}^{\{n\}} = \left(K_{0, h}^{\{n\}}, K_{1, h}^{\{n\}}, ..., K_{h, h}^{\{n\}} \right)$$
be a variable encoding a random partition of $[h] = \{1, ..., h\}$, with $K_{0, h}^{\{n\}} : = K_{h}^{\{n\}}$ the number of blocks and for $r \ge 1$, $K_{r, h}^{\{n\}}$ the number of blocks of numerosity $r$.  Defining, for $r \in \{0, ..., n\}$,  
\begin{equation} \label{xi_r}
\xi_{r, h}^{\{n\}} = K_{r, h}^{\{n\}}  -K_{r, h-1}^{\{n\}}
\end{equation} 
and letting
$$
\mathcal{F}_{h-1}^{\{n\}} = \sigma\left(\mathbf{K}_1^{\{n\}}, ..., \mathbf{K}_{h-1}^{\{n\}}\right), 
$$
for $h \in \{2, ..., n\}$, the following sequential relation holds:
\begin{equation}\label{CRP_arr}
P \left[ \xi_{r, h}^{\{n\}}  =  k\, | \, \mathcal{F}_{h-1}^{\{n\}}\right] = \begin{cases}
p_{h-1}^{\{n\}} & \text{ if } k=1  \\
q_{h-1}^{\{n\}} & \text{ if } k=-1 \\
1- p_{h-1}^{\{n\}} - q_{h-1}^{\{n\}} & \text{ if } k=0
\end{cases}
\end{equation}
with 
\begin{equation}\label{p}
 p_{r, h-1}^{\{n\}} = \begin{cases}
\frac{\alpha K_{h-1}^{\{n\}} + \lambda n}{\lambda n + h-1} & \text{ if } r =0, 1\\
\frac{(r-1-\alpha) K_{r-1, h-1}^{\{n\}}}{\lambda n + h-1} & \text{ if } r \ge 2
\end{cases}
\end{equation}
and
\begin{equation}\label{q}
 q_{r, h-1}^{\{n\}} = \begin{cases}
0 & \text{ if } r =0\\
\frac{(r-\alpha) K_{r, h-1}^{\{n\}}}{\lambda n + h-1} & \text{ if } r \ge 1
\end{cases}
\end{equation}

Clearly, \eqref{CRP_arr} is recovered by substituting $\theta = \lambda n$ in the equivalent relation \eqref{CRP} of the standard CRP construction. It should be noted, however, that in this regime the recursive relation only holds for $h \le n$, and the array $\left(\mathbf{K}_1^{\{n\}}, ..., \mathbf{K}_{n}^{\{n\}}\right)$ is not further extendable to an exchangeable random partition of $[n+1]$. Furthermore, two partitions $\mathbf{K}_{i}^{\{n\}}$ and $\mathbf{K}_{j}^{\{m\}}$ with $m \neq n$ do not satisfy compatibility relations of any kind.  In other words, we are considering triangular arrays of the form
\begin{align*}
(\theta = \lambda) \quad& \textcolor{black}{\mathbf{K}_1^{{\{1\}}}} \\
(\theta = 2\lambda) \quad & \mathbf{K}_1^{\{2\}} \, \textcolor{black}{\mathbf{K}_2^{\{2\}}}\\
\vdots \quad \quad \quad &  \quad \quad \quad   \quad \quad     \ddots\\
(\theta = n\lambda) \quad & \mathbf{K}_1^{\{n\}} \,   \mathbf{K}_2^{\{n\}} \,  ... \,  \mathbf{K}_{n-1}^{\{n\}} \,  \textcolor{black}{\mathbf{K}_{n}^{\{n\}}}  \quad \quad \\
\vdots \quad \quad \quad &  \qquad  \qquad   \qquad   \qquad \qquad   \ddots
\end{align*}
%
%
and Theorem \ref{main_d} is concerned with the ``diagonal'' sequence $\left(\mathbf{K}_{d,n}^{\{n\}}\right)_{n}$ where for each $n$, $\mathbf{K}_{d,n}^{\{n\}}$ is the projection of $\mathbf{K}_{n}^{\{n\}}$ on the first $d+1$ coordinates.

\subsection{An important lemma}
This result is an essential preliminary to the proof of theorem \ref{main_d}. 
\begin{lem} \label{Kxrnn}
For $x \in [0, 1]$, define 
\begin{displaymath}
m_{0; \alpha, \lambda}(x) 
= \begin{cases}\frac{\lambda}{\alpha} \left[ \left(\frac{\lambda +x}{\lambda}\right)^\alpha - 1 \right] &  \text{if } \alpha \in (0, 1)\\
\lambda \log \left(\frac{\lambda +x}{\lambda}\right) & \text{if } \alpha = 0
\end{cases}
\end{displaymath}
and for $ r \ge 1$,
\begin{displaymath}
m_{r; \alpha, \lambda}(x)   = \frac{(1-\alpha )_{\uparrow r-1}}{r! } \,  x^r  \lambda^{1-\alpha} (x+\lambda)^{\alpha -r}
\end{displaymath}
Then, as $n \to +\infty$, for all $r \in \{0, ...,n\}$,
\begin{enumerate}[(i)] 
\item 
\begin{equation*}
\frac{\mathbb{E}\left[K_{r, \lfloor xn \rfloor }^{\{n\}} \right]}{n} \rightarrow m_{r; \alpha, \lambda}(x) 
\end{equation*}
uniformly on $[0,1]$.
\item the following SLLN holds uniformly for $x \in [0,1]$:
\begin{equation}\label{slln}
\frac{K_{r, \lfloor xn \rfloor} ^{\{n\}}}{n} \stackrel{a.s.}{\longrightarrow} m_{r; \alpha, \lambda}(x) 
\end{equation}
\end{enumerate}
\end{lem}
\begin{proof}
The structure of the proof is the same for the three cases $r \ge 1$ ; $r = 0,\,  \alpha = 0$ and $r =0, \, \alpha \in (0,1)$, but it involves different calculations. For this reason, we present  here only the general idea; the complete calculations, differentiated for the three cases, are deferred to appendix \ref{app3_1}. \\

For all $r \in \{0, ..., n\}, \alpha \in [0,1)$ and  $ \lambda>0$, we exploit the explicit expressions for 
$$ \mathbb{E} \left[\left(K_{r, h} ^{\{n\}}\right)_{\downarrow s}\right]. $$
They are available for all $s \in \mathbb{N}$ in \cite[Proposition 1]{Fav(13)} for $r \ge 1$ and can be derived from \cite[Equation 3.11]{Pit(06)} for $r=0$. Upon substituting $h$ with $\lfloor xn \rfloor$, standard asymptotical results allow to prove by direct calculation that 
\begin{equation}\label{falling_expansion}
\lim_{n \to \infty}  \sup_{x \in [0,1]}\frac{1}{n^{s-2}} \cdot \left| \mathbb{E} \left[(K_{r, \lfloor xn \rfloor} ^{\{n\}})_{\downarrow s}\right] - \left(n^{s} \cdot \left(m_{r; \alpha, \lambda}(x)\right)^s +n^{s-1} \cdot \mathcal{S}_s(x) \right)\right|  = c
 \end{equation}
for some constant $c$ and some explicit function $\mathcal{S}_s(x)$ not depending on $n$. In particular, for $s = 1$, this rewrites as
$$
\lim_{n \to + \infty} \sup_{x \in [0,1]} \left| \frac{\mathbb{E} \left[K_{ r, \lfloor xn \rfloor }^{\{n\}}\right]}{n} -   m_{r;  \alpha, \lambda}(x)  \right| =0 
$$
Proving point (i).

To prove the strong law of large numbers \eqref{slln}, we show complete convergence of the sequence $K_{r, \lfloor xn \rfloor} ^{\{n\}} / n$ by showing that $\mathbb{E}\left[ \left(K_{r, \lfloor xn \rfloor} ^{\{n\}} - \mathbb{E}\left[K_{r, \lfloor xn \rfloor} ^{\{n\}}\right]\right)^4\right]   = O(n^2) $ uniformly for $x \in [0,1]$. Write
\begin{align*}
&\mathbb{E}\left[ \left(K_{r, \lfloor xn \rfloor} ^{\{n\}} - \mathbb{E}\left[K_{r, \lfloor xn\rfloor}^{\{n\}}\right]\right)^4\right]  \\
& \quad = \sum_{k  = 0}^4 (-1)^{4-k} \binom{4}{k} \left(\mathbb{E}\left[K_{r, \lfloor xn\rfloor}^{\{n\}}\right]\right)^{4-k} \sum_{s=0}^k S(k, s) \, \mathbb{E} \left[\left(K_{r, \lfloor xn \rfloor }^{\{n\}}\right)_{\downarrow s} \right] , 
\end{align*}
where $S(k, s)$ denote the Stirling number of the second kind. Expanding such expression and computing the expansion of powers of $\mathbb{E}\left[K_{r, \lfloor xn \rfloor }^{\{n\}} \right]$ yields, after tedious but straightforward calculations, 
$$
\lim_{n \to +\infty} \sup_{x \in [0,1]} \frac{1}{n^2} \, \left| \mathbb{E}\left[ \left(K_{r, \lfloor xn \rfloor} ^{\{n\}} - \mathbb{E}\left[K_{r, \lfloor xn\rfloor}^{\{n\}}\right]\right)^4\right] - \left( n^4 \cdot \mathcal{A}_{r; \alpha, \lambda}(x) + n^3 \cdot \mathcal{B}_{r; \alpha, \lambda}(x) \right) \right| = c
$$
for some constant $c$,  where
$$
\mathcal{A}_{r; \alpha, \lambda}(x)  =\left( m_{r; \alpha, \lambda}(x)  \right)^4  \cdot \left(1-4+6-3 \right)  = 0 
$$
\begin{align*}
\mathcal{B}_{r; \alpha, \lambda}(x)  &=m^3_{r; \alpha, \lambda}(x) \left[(-12+12-4) \mathcal{S}_1(x)  + (6-12+6) \right]  + 6\, m^2_{r; \alpha, \lambda}(x) \, \mathcal{S}_2(x) \\
& \quad + m_{r; \alpha, \lambda}(x)\, \mathcal{S}_3(x) + \mathcal{S}_4(x)\\
&=-4 \, m^3_{r; \alpha, \lambda}(x) \, \mathcal{S}_1(x)    + 6\, m^2_{r; \alpha, \lambda}(x) \, \mathcal{S}_2(x)  -4 m_{r; \alpha, \lambda}(x)\, \mathcal{S}_3(x) + \mathcal{S}_4(x).
\end{align*}
Further, it can be proven by substituting the explicit expressions for $S_s(x)$ that
\begin{equation}\label{Bis0}
 \mathcal{B}_{r; \alpha, \lambda}(x) =0 
 \end{equation}
for all $r \in \{0, ..., n\}$ and $\alpha \in [0,1)$;  see appendix \ref{app3_1} for details. 
Therefore, 
$$\lim_{n \to +\infty} \sup_{x \in [0,1]} \frac{1}{n^2} \,  \mathbb{E}\left[ \left(K_{r, \lfloor xn \rfloor} ^{\{n\}} - \mathbb{E}\left[K_{r, \lfloor xn \rfloor} ^{\{n\}}\right]\right)^4\right]   = c $$ 
and the proof is concluded.
\end{proof}

\subsection{Proof of the expansion \eqref{mom_m_d} and of the SLLN \eqref{lln_d}}
Upon noting that, for all $r \ge 0$, 
\begin{displaymath}
\mathfrak{m}_{r; \alpha, \lambda} = m_{r; \alpha, \lambda}(1)  
\end{displaymath}
apply Lemma \ref{Kxrnn} for $x = 1$. 

\subsection{Proof of the CLT \eqref{clt_d}}
The proof of  \eqref{clt_d} relies on the construction of a martingale and on the use of the Lindeberg--Levy central limit theorem for arrays of martingales given in \cite[Corollary 3.1]{HH}.

\subsubsection{Construction of the martingale}
The sequential construction of the partition under the large-$\theta$ regime of Section \ref{CRP_arrays_sec} entails that for all $h \in \{1,\ldots,n\}$ and for all $r \geq 0$
\begin{align}\label{xi_mean_r}
\nonumber \mathbb{E}\left[K_{r,h}^{\{n\}} \, | \, \mathcal{F}_{h}^{\{n\}}\right] & = K_{r, h-1}^{\{n\}} + p_{r, h-1}^{\{n\}} - q_{r, h-1}^{\{n\}}\\
&= \gamma_{r, h-1}^{\{n\}}  K_{r,h-1}^{\{n\}}  + \beta_{r, h-1}^{\{n\}}
\end{align}
where 
$$
\gamma_{r, h-1}^{\{n\}} = \frac{\lambda n + h-1-r+\alpha}{\lambda n + h-1}
$$
and
$$
\beta_{r, h-1}^{\{n\}}  = \begin{cases} \frac{\lambda n}{\lambda n +h-1} & \text{ if } r = 0\\
  p_{r, h-1}^{\{n\}}&\text{ if } r \ge 1
\end{cases}.
$$
For $r \in \{0, ..., n\}$  and for $n \in \mathbb{N}$ and $h \in \{1, ..., n\}$, define 
\begin{equation}\label{an_r}
a_{r,h}^{\{n\}} = \prod_{k=1}^{h-1} \left(\gamma_{r,k}^{\{n\}}\right)^{-1}
\end{equation}
and
\begin{equation}\label{A_n}
A_{r,h}^{\{n\}} = \sum_{k=1}^{h-1} a_{r, k+1}^{\{n\}} \beta_{r,k}^{\{n\}}.
\end{equation}
Further, for fixed  $d \in \mathbf{N}$,  let 
\begin{align*}
\mathbf{\Xi}_{d, h}^{\{n\}} & = \left(\xi_{0,h}^{\{n\}} , \xi_{1, h}^{\{n\}}, ..., \xi_{d, h}^{\{n\}}\right)^T\\
\mathbf{A}_{d, h}^{\{n\}} & = \left(A_{0,h}^{\{n\}} , A_{1, h}^{\{n\}}, ..., A_{d, h}^{\{n\}}\right)^T\\ 
\mathbf{a}_{d, h}^{\{n\}}  &= \left(a_{0,h}^{\{n\}} , a_{1, h}^{\{n\}}, ..., a_{d, h}^{\{n\}}\right)^T. 
\end{align*}
We introduce the following notation: given a vector $v \in \mathbb{R}^{d+1}$, we denote by $\operatorname{Diag}(v)$ the  $(d+1) \times (d+1)$ diagonal matrix defined by
\begin{align*}
\left(\operatorname{Diag}(v) \right)_{i, j}  &= \begin{cases}
0 & \text{ if } i \neq j\\
v_i & \text{ if } i = j
\end{cases}
\end{align*}
With this convention, let 
$$
\mathcal{A}_{d, h}^{\{n\}} = \operatorname{Diag}\left(\mathbf{a}_{d, h}^{\{n\}}\right)
$$
and finally define 
\begin{equation}\label{bigM}
\mathbf{M}_{d, h}^{\{n\}}  = \mathcal{A}_{d, h}^{\{n\}} \cdot \mathbf{K}_{d,h}^{\{n\}}  - \mathbf{A}_{d, h}^{\{n\}} 
\end{equation}
and
\begin{equation}\label{D}
\mathbf{\Delta}_{d, h}^{\{n\}} = \mathbf{M}_{d, h}^{\{n\}} - \mathbf{M}_{d, h-1 }^{\{n\}}
\end{equation}
\begin{lem}\label{mart_lem}
For every $d \in \mathbb{N}$ and $n \in \mathbb{N}$, $\left(\mathbf{M}_{d, h}^{\{n\}}\right)_{1\leq h \leq n}$ is a martingale with respect to the filtration $\left(\mathcal{F}_h^{\{n\}}\right)_{1\leq h \leq n}$.
\end{lem}
The proof of this lemma is immediate and relies on the identity 
\begin{equation}\label{deltaform_d}
\mathbf{\Delta}_{d, h}^{\{n\}} = \mathcal{A}_{d, h}^{\{n\}}  \cdot  \left[ \mathbf{\Xi}_{d, h}^{\{n\} }- \left(p_{r, h-1}^{\{n\}} - q_{r, h-1}^{\{n\}}\right)_{r \in \{0, ..., d\}} \right].
\end{equation}
See appendix \ref{sec_martingale} for details. 
\subsubsection{Completing the proof }
Let $n$ be a sequence of integers such that $h_n \to + \infty$ as $n \to +\infty$. Denote by
$$
\langle \mathbf{M}_d^{\{n\}} \rangle = \left(\langle \mathbf{M}_d^{\{n\}} \rangle_1, ..., \langle \mathbf{M}_d^{\{n\}} \rangle_n  \right)^T
$$ 
the increasing process of $ \mathbf{M}_d^{\{n\}}$, i.e.
\begin{equation}\label{increasingproc_d}
\langle  \mathbf{M}_d^{\{n\}} \rangle_{h_n} =
\sum_{h = 1}^{h_n} \mathbb{E} \left[ \mathbf{\Delta}_{d,h}^{\{n\}} \cdot \left(\mathbf{\Delta}_{d,h}^{\{n\}}\right)^T \, \bigg| \, \mathcal{F}_{h-1}^{\{n\}}\right].
\end{equation}
The following lemma is an immediate consequence of \cite[Corollary 3.1]{HH}: 
%
%
\begin{lem} \label{increasingproc_lem_d} Assume that the following two conditions are satisfied,
\begin{enumerate}[(a)]
\item $$ \frac{ \langle \mathbf{M}_d^{\{n\}} \rangle_{h_n}}{h_{n}} \stackrel{a.s.}{\longrightarrow }G_{d, \alpha, \lambda}$$
 for some positive semi-definite $(d+1) \times (d+1)$ matrix $G_{d, \alpha, \lambda}$
\item For every $\varepsilon > 0$,
$$\frac{1}{h_n} \, \sum_{h = 1}^{h_n} \mathbb{E} \left[ \left\lVert \mathbf{\Delta}_{d,h}^{\{n\}}\right\rVert^2 \cdot \mathds{1}_{\left\{ \left\lVert \mathbf{\Delta}_{d,h}^{\{n\}}\right\rVert^2 >h_n \, \varepsilon \right\}} \, \bigg| \, \mathcal{F}_{h-1}^{\{n\}}\right] \stackrel{a.s.}{\longrightarrow }0 $$
\end{enumerate}
Then, we have
\begin{equation}
\label{CLTMG1}
\frac{\mathbf{M}_{d, h_n}^{\{n\}}}{\sqrt{h_n}}\stackrel{w}{\longrightarrow} N(\mathbf{0}, G_{d, \alpha, \lambda})
\end{equation}
\end{lem} 
We now show that, for the sequence $h_n = n$,  conditions (a) and (b) of lemma \ref{increasingproc_lem_d} are satisfied with $G_{d, \alpha, \lambda} = \Gamma_{d, \alpha, \lambda}$ where $\Gamma_{d, \alpha, \lambda}$ is as in \eqref{Gamma_d}. 
To prove (a), we work component by component. It follows from \eqref{deltaform_d} that
\begin{align*}
&\left(\mathbf{\Delta}_{d,h}^{\{n\}} \cdot \left(\mathbf{\Delta}_{d,h}^{\{n\}}\right)^T\right)_{i, j} \\
& \quad = a_{i-1, h}^{\{n\}} a_{j-1,h}^{\{n\}} \cdot \left(\xi_{i-1, h}^{\{n\}} -p_{i-1, h-1}^{\{n\}} + q_{i-1, h-1}^{\{n\}}\right) \left(\xi_{j-1, h}^{\{n\}} -p_{j-1, h-1}^{\{n\}} + q_{j-1, h-1}^{\{n\}}\right)\\
& \quad = a_{i-1, h}^{\{n\}} a_{j-1,h}^{\{n\}} \cdot \left[\xi_{i-1, h}^{\{n\}}\xi_{j-1, h}^{\{n\}}  -\left(p_{i-1, h-1}^{\{n\}} - q_{i-1, h-1}^{\{n\}} \right)\xi_{j-1, h}^{\{n\}}  \right.\\
& \quad \quad \left. - \left(p_{j-1, h-1}^{\{n\}} - q_{j-1, h-1}^{\{n\}} \right) \xi_{i-1, h}^{\{n\}} + \left(p_{i-1, h-1}^{\{n\}} - q_{i-1, h-1}^{\{n\}} \right) \left(p_{j-1, h-1}^{\{n\}} - q_{j-1, h-1}^{\{n\}} \right)\right]
\end{align*}
so that
\begin{displaymath}
 \left(\mathbb{E} \left[ \mathbf{\Delta}_{d,h}^{\{n\}} \cdot \left(\mathbf{\Delta}_{d,h}^{\{n\}}\right)^T \, \bigg| \, \mathcal{F}_{h-1}^{\{n\}}\right]\right)_{i,j} = a_{i-1, h}^{\{n\}} a_{j-1,h}^{\{n\}} \cdot \left[ \left(P_{h-1}^{\{n\}}\right)_{i, j}  - \left(R_{h-1}^{\{n\}}\right)_{i, j} \right]
 \end{displaymath}
where, for all $1\leq i \leq j$, 
$$
\left(P_{h-1}^{\{n\}}\right)_{i, j} = \mathbb{E} \left[\xi_{i-1, h}^{\{n\}}\xi_{j-1, h}^{\{n\}}  \, \bigg| \, \mathcal{F}_{h-1}^{\{n\}} \right] = \begin{cases}
0 & \text{ if } j \ge i+2\\
-q_{i-1, h-1}^{\{n\}} & \text{ if }j = i + 1 \text{ and } i \ge 2\\
p_{i-1, h-1}^{\{n\}} + q_{i-1, h-1}^{\{n\}} & \text{ if } j = i \text { or } i = 1, j=2
\end{cases}
$$
and
$$
\left(R_{h-1}^{\{n\}}\right)_{i, j} = \left(p_{i-1, h-1}^{\{n\}} - q_{i-1, h-1}^{\{n\}} \right) \left(p_{j-1, h-1}^{\{n\}} - q_{j-1, h-1}^{\{n\}} \right).
$$ 
This implies
\begin{align*}\label{incp}
\nonumber \left(\langle  \mathbf{M}_d^{\{n\}} \rangle_{n} \right)_{i,j} &= \sum_{h = 1}^{n} a_{i-1, h}^{\{n\}} a_{j-1,h}^{\{n\}} \cdot \left[ \left(P_{h-1}^{\{n\}}\right)_{i, j}  - \left(R_{h-1}^{\{n\}}\right)_{i, j} \right]\\
& \nonumber = \sum_{h = 1}^{n} a_{i-1, n \cdot \frac{h-1}{n} + 1}^{\{n\}} a_{j-1,n \cdot \frac{h-1}{n} + 1}^{\{n\}} \cdot \left[ \left(P_{n \cdot \frac{h-1}{n}}^{\{n\}}\right)_{i, j}  - \left(R_{n \cdot \frac{h-1}{n}}^{\{n\}}\right)_{i, j} \right]\\
& =n \cdot  \sum_{h = 1}^n a_{i-1, n \cdot \frac{h-1}{n} + 1}^{\{n\}} a_{j-1,n \cdot \frac{h-1}{n} + 1}^{\{n\}} \cdot \left[ \left(P_{n \cdot \frac{h-1}{n}}^{\{n\}}\right)_{i, j}  - \left(R_{n \cdot \frac{h-1}{n}}^{\{n\}}\right)_{i, j} \right] \cdot \frac{1}{n}
\end{align*}
so that $ \left(\langle  \mathbf{M}_d^{\{n\}} \rangle_{n} \right)_{i,j}  / n$ can be viewed as a Riemann sum for the function $F^{\{n\}}_{i, j}$ on $[0,1]$, 
where
\begin{equation}\label{Fnij}
F^{\{n\}}_{i, j}(x)  = a_{i-1, \lfloor n x\rfloor  + 1}^{\{n\}} a_{j-1, \lfloor n x\rfloor  + 1}^{\{n\}} \cdot \left[ \left(P_{ \lfloor n x\rfloor }^{\{n\}}\right)_{i, j}  - \left(R_{ \lfloor n x\rfloor }^{\{n\}}\right)_{i, j} \right]
\end{equation}
Making use of lemma \ref{Kxrnn} 
we are able to show that
$$F^{\{n\}}_{i,j} (x) \stackrel{a.s.}{\longrightarrow} f_{i,j} (x) $$
uniformly for $x \in [0,1]$, where $f_{i,j} (x) $ is a deterministic, explicit function only depending on $\alpha, \lambda, i$ and $j$, and it is Riemann integrable on [0,1]. By definition and elementary properties of the Riemann integral, 
\begin{displaymath}
 \frac{\left(\langle  \mathbf{M}_d^{\{n\}} \rangle_{n} \right)_{i,j}}{n}  \stackrel{a.s.}{\longrightarrow}  \int_{0}^1 f_{i, j}(x) \, \mathrm{d}x.
\end{displaymath}
and since we are able to prove
\begin{align*}
 \int_{0}^1 f_{i, j}(x) \, \mathrm{d}x =  \left(\Gamma_{d, \alpha,\lambda}\right)_{i,j},
\end{align*}
the proof is concluded.  See appendix \ref{app3_2} for the proof of the convergence of $F_{i,j}^{\{n\}}$, the definition of $f_{i,j}$ and the computation of the integral.\\

For what concerns condition (b), we can write
\begin{align*}
S_{n}& :=  \frac{1}{n}   \, \sum_{h = 1}^{n} \mathbb{E} \left[ \left\lVert \mathbf{\Delta}_{d,h}^{\{n\}}\right\rVert^2 \cdot \mathds{1}_{\left\{ \left\lVert \mathbf{\Delta}_{d,h}^{\{n\}}\right\rVert^2 >n  \varepsilon \right\}} \, \bigg| \, \mathcal{F}_{h-1}^{\{n\}}\right] \\
 & =  \frac{1}{n} \sum_{h = 1}^{n} \mathbb{E} \left[ \frac{\left\lVert \mathbf{\Delta}_{d,h}^{\{n\}}\right\rVert^4}{\left\lVert \mathbf{\Delta}_{d,h}^{\{n\}}\right\rVert^2} \cdot \mathds{1}_{\left\{ \left\lVert \mathbf{\Delta}_{d,h}^{\{n\}}\right\rVert^2 >n  \varepsilon \right\}} \, \bigg| \, \mathcal{F}_{h-1}^{\{n\}}\right] 
 \end{align*}
 and since 
 $$ 
 \frac{\mathds{1}_{\left\{ \left\lVert \mathbf{\Delta}_{d,h}^{\{n\}}\right\rVert^2 >n   \varepsilon \right\}} }{\left\lVert \mathbf{\Delta}_{d,h}^{\{n\}}\right\rVert^2}  \le \frac{1}{n\varepsilon}
 $$
 almost surely, then
 \begin{align*}
S_{n}&  \le \frac{1}{n^2  \, \varepsilon} \, \sum_{h = 1}^{n} \mathbb{E} \left[ \left\lVert \mathbf{\Delta}_{d,h}^{\{n\}}\right\rVert^4 \, \bigg| \, \mathcal{F}_{h-1}^{\{n\}}\right]\\
&  = \frac{1}{n^2  \, \varepsilon} \, \sum_{h = 1}^{n}   \mathbb{E} \left[ \left\lVert \mathcal{A}_{d, h}^{\{n\}}  \cdot  \left[ \mathbf{\Xi}_{d, h}^{\{n\} }- \mathbb{E} \left[\mathbf{\Xi}_{d, h}^{\{n\} } \, | \, \mathcal{F}_{h-1}^{\{n\}} \right]\right]\right\rVert^4 \, \bigg| \, \mathcal{F}_{h-1}^{\{n\}}\right]\\
&\le \frac{\mathcal{C}}{n^2  \, \varepsilon} \, \sum_{h = 1}^{n}   \mathbb{E} \left[ \left\lVert \mathcal{A}_{d, h}^{\{n\}}  \right\rVert^4  \cdot \left\lVert \mathbf{\Xi}_{d, h}^{\{n\} }- \mathbb{E} \left[\mathbf{\Xi}_{d, h}^{\{n\} } \, | \, \mathcal{F}_{h-1}^{\{n\}} \right] \right\rVert^4 \, \bigg| \, \mathcal{F}_{h-1}^{\{n\}}\right]
\end{align*}
for some positive constant $\mathcal{C}$. Finally, observe that, since the $\xi_{r,h}^{\{n\}}$ take values in $\{-1, 0, 1\}$, 
$$
\left \lVert\ \mathbf{\Xi}_{d, h}^{\{n\} }  - \mathbb{E} \left[\mathbf{\Xi}_{d, h}^{\{n\} } \, | \, \mathcal{F}_{h-1}^{\{n\}} \right] \right\rVert \le 2
$$
almost surely; furthermore, by definition of $\mathcal{A}_{d,h}^{\{n\}}$, 
$$
 \left\lVert \mathcal{A}_{d, h}^{\{n\}}  \right\rVert  =  \left\lVert \mathbf{a}_{d, h}^{\{n\}}  \right\rVert.
$$
Then, we can bound $S_{n}$ by 
\begin{align*}
S_{n}&\le \frac{16\,  \mathcal{C}}{n^2   \varepsilon} \, \sum_{h = 1}^{n}   \left\lVert \mathbf{a}_{d, h}^{\{n\}}  \right\rVert^4  \\
&\le \frac{16\,  \mathcal{C}}{n  \varepsilon} \cdot \max_{r \in \{0, ..., d\}, \, x \in [0,1]}  \left(a_{r, \lfloor nx \rfloor }^{\{n\}}\right)^4   \\
& =: \mathfrak{S}_{n}
\end{align*}
 We know from 
(i) of lemma \ref{as_1_d} that for every $r \in \{1, ..., d\}$ 
$$a_{r, \lfloor n x\rfloor  + 1}^{\{n\}}  =  \left(\frac{\lambda + x}{\lambda}\right)^{r-\alpha} + O\left(\frac{1}{n}\right) \le  \left(\frac{\lambda + 1}{\lambda}\right)^{d-\alpha} + O\left(\frac{1}{n}\right)$$
and for $r = 0$, 
$$a_{r, \lfloor n x\rfloor  + 1}^{\{n\}}  =  \left(\frac{\lambda}{\lambda+x }\right)^{1+\alpha} + O\left(\frac{1}{n}\right) \le 1+   O\left(\frac{1}{n}\right) $$
uniformly for $x \in [0,1]$. In conclusion
$$\mathfrak{S}_{n} = \frac{16\, \mathcal{C}}{n  \varepsilon} \cdot \left(\frac{\lambda + 1}{\lambda}\right)^{4(d-\alpha)} + O\left(\frac{1}{n^2}\right) = O\left(\frac{1}{n}\right)$$
and this proves point (b). \\

Therefore, it follows from Lemma \ref{increasingproc_lem_d} that  
as $n \to +\infty$,
\begin{equation}
\label{CLTMG}
\frac{\mathbf{M}_{d, n}^{\{n\}}}{\sqrt{n}}\stackrel{w}{\longrightarrow} N(\mathbf{0}, \Gamma_{d, \alpha, \lambda})
\end{equation}
By definition of $\mathbf{M}_{d,n}^{\{n\}}$ \eqref{bigM}, 
\begin{align*}
\mathbf{K}_{d, n}^{\{n\}} & = \left( \mathcal{A}_{d, n}^{\{n\}}\right)^{-1}  \cdot \left[ \mathbf{M}_{d, n}^{\{n\}} + \mathbf{A}_{d, n}^{\{n\}} \right]\\
& = \left( \mathcal{A}_{d, n}^{\{n\}}\right)^{-1}  \cdot \mathbf{M}_{d, n}^{\{n\}} + \mathfrak{M}_{d, n}^{\{n\}}  + \left( \mathcal{A}_{d, n}^{\{n\}}\right)^{-1}  \cdot  \mathbf{A}_{d, n}^{\{n\}} - \mathfrak{M}_{d, n}^{\{n\}} 
\end{align*}
It follows from lemma \ref{Kxrnn} that
$$
\left( \mathcal{A}_{d, n}^{\{n\}}\right)^{-1}  \cdot  \mathbf{A}_{d, n}^{\{n\}} \stackrel{a.s.}{\longrightarrow} \mathfrak{M}_{d, n}^{\{n\}}, 
$$ 
and by (i) of lemma \ref{as_1_d}
$$ 
\left(\mathcal{A}_{d, n}^{\{n\}} \right) ^{-1} \stackrel{a.s.}{\longrightarrow} \mathfrak{A}_{d, \alpha, \lambda}
$$
where 
$$
\mathfrak{A}_{d, \alpha, \lambda} = \operatorname{Diag}\left(\left(\left[\frac{\lambda+1}{\lambda}\right]^{r-\alpha}\right)_{r \in \{0, ..., d\}}\right).
$$
In conclusion, apply Slutsky's lemma to show that
$$
 \frac{\mathbf{K}_{d, n}^{\{n\}} - \mathfrak{M}_{d, m}^{\{n\}} }{\sqrt{n}}\stackrel{w}{\longrightarrow} N(\mathbf{0}, \mathfrak{A}_{d, \alpha, \lambda}^T \cdot \Gamma_{d, \alpha, \lambda}\cdot \mathfrak{A}_{d, \alpha, \lambda})
$$
and since by definition \eqref{Sigma_d}
$$
\mathfrak{A}_{d, \alpha, \lambda}^T \cdot \Gamma_{d, \alpha, \lambda}\cdot \mathfrak{A}_{d, \alpha, \lambda} = \Sigma_{d, \alpha, \lambda}
$$
the proof is concluded.





\section*{Acknowledgments}
C.C. is grateful to B.B. and the Institut de Math\'ematiques de Bordeaux for their hospitality during the visit that led to this collaboration. 


\appendix 
\renewcommand{\theequation}{A.\arabic{equation}}
\renewcommand{\thesubsection}{A.\arabic{subsection}}
\setcounter{thm}{0}
\renewcommand{\thethm}{A.\arabic{thm}}

\section*{Appendix}\label{app3}
\addcontentsline{toc}{section}{Appendix}

\subsection{The Chinese Restaurant Process}\label{app_pyp}
The sequential construction of the Ewens--Pitman model was introduced in the seminal work of \cite[Proposition~9]{Pit(95)}.
For $\alpha \in [0,1)$ and $\theta > -\alpha$, it is possible to show that the following recursive construction yields  an exchangeable random partition of the set $[n]$.
\begin{paragraph}
\noindent Conditionally on the number of blocks $K_n = k$ and on the partition subsets  (the ``tables'', in the restaurant metaphor) $\{T_1,\ldots,T_k\}$ with corresponding sizes $(n_1,\ldots,n_k)$, the partition of $[n+1]$
is obtained by extending that of $[n]$ in such a way that the element $n+1$ is assigned
to an existing subset $T_i$, for $1 \le i \le k$, with probability
\[
\frac{n_i - \alpha}{n + \theta},
\]
or initiates a new subset with probability
\[
\frac{\alpha k + \theta}{n + \theta}.
\]
Since $n = n_1 + \cdots + n_k$, we clearly have
\begin{equation*}
\frac{1}{n+\theta}\sum_{i=1}^{k} (n_i - \alpha) + \frac{\alpha k + \theta}{n+\theta}
= \frac{n - \alpha k + \alpha k + \theta}{n + \theta} = 1.
\end{equation*}
From now on, for $K_n \in \{1,\ldots,n\}$, let
\[
\mathbf{N}_n = (N_{1,n},\ldots,N_{K_n,n})
\]
denote the sizes of the partition subsets $\{T_1,\ldots,T_{K_n}\}$. As shown in
\cite[Proposition~9]{Pit(95)}, the above construction yields the joint distribution
\eqref{epsm} of the random vector $(K_n,\mathbf{N}_n)$.
\end{paragraph}

\begin{paragraph}
\noindent For $r = 1, ..., n$, we introduce the notation $K_{r, n}$  to denote the number of tables with given numerosity $r$. Note that the above construction immediately yields a recursive structure not only for $K_n$, but also for the $K_{r, n}$. In fact, letting
$$
\mathbf{K}_n = \left(K_n, K_{1,n}, ..., K_{n,n}\right)
$$
and
$$
\mathcal{F}_{n} = \sigma\left(\mathbf{K}_1, ..., \mathbf{K}_n\right)
$$
there hold
\[
P(K_{n+1} = K_n + 1 \mid \mathcal{F}_{n} ) = \frac{\alpha K_n + \theta}{n + \theta}.
\]
and  for $r \ge 1$, 
\begin{equation}\label{CRP}
P \left[ K_{r, n+1}  = K_{r, n} + k\, | \, \mathcal{F}_{n}\right] = \begin{cases}
p_{r, n} & \text{ if } k=1  \\
q_{r, n}& \text{ if } k=-1 \\
1- p_{r, n} - q_{r, n} & \text{ if } k=0
\end{cases}
\end{equation}
with 
$$
 p_{r, n} = \begin{cases}
\frac{\alpha K_n + \theta }{\theta+ n} & \text{ if } r = 1\\[2pt]
\frac{(r-1-\alpha) K_{r-1, n}}{\theta + n} & \text{ if } r \ge 2
\end{cases}
$$
and 
$$
 q_{r, h-1}^{\{n\}} = \frac{(r-\alpha) K_{r, n}}{\theta + n} 
$$
\end{paragraph}


\subsection{Details on the proof of lemma \ref{Kxrnn}}\label{app3_1}
\subsubsection{$r \ge 1$, $\alpha \in [0,1)$}

From \cite[Proposition 1]{Fav(13)} we recover the following expression for the (falling) factorial moments of $K_{r, \lfloor xn \rfloor}^{\{n\}}$:
\begin{align*}
\mathbb{E} \left[\left(K_{r, \lfloor xn \rfloor }^{\{n\}}\right)_{\downarrow s} \right] & = \left(\frac{(1-\alpha)_{\uparrow r-1}}{r!} \right)^s \, \frac{\Gamma(xn + 1 -\mathfrak{r}_x )}{\Gamma(xn -rs + 1 -\mathfrak{r}_x )} \, \frac{\Gamma(\lambda n+1)}{\Gamma((\lambda +x ) n - \mathfrak{r}_x)} \times \\
& \quad \times \alpha^{s-1} \, \frac{\Gamma\left(\frac{\lambda n}{\alpha} +s\right)}{\Gamma\left(\frac{\lambda n}{\alpha} +1\right)} \, \frac{\Gamma((\lambda +x ) n + s\alpha -sr - \mathfrak{r}_x)}{\Gamma(\lambda n + s\alpha)}
\end{align*}
where we are denoting by $\mathfrak{r}_x = xn - \lfloor xn \rfloor$. Since $0 \le \mathfrak{r}_x < 1$ for all $x \in [0, 1]$ and for all $n \in \mathbb{N}$, we can apply the asymptotic expansion for ratios of gamma functions \cite[Equation 1]{TE(51)}: as $|z |\to +\infty$, 
$$
\frac{\Gamma(z+\beta)}{\Gamma(z+\gamma) } = z^{\beta -\gamma}\cdot  \left[1+ \frac{(\beta - \gamma)(\beta+\gamma-1)}{2z} + O\left(|z|^{-2}\right) \right]
$$
This entails that \eqref{falling_expansion} holds with 
\begin{displaymath}
\mathcal{S}_s(x) =  m^s_{r; \alpha, \lambda}(x)   \cdot   [s^2 A(x) + s B(x)]
\end{displaymath}
where
\begin{align*}
A (x)& =  \frac{x^2 \alpha (1-\alpha) + x [\lambda \alpha - 2r\alpha \lambda  ] -\lambda^2r^2 }{2x\lambda (x+\lambda)}\\
B(x) & =  \frac{x\left[r(1-2\mathfrak{r}_x) -\alpha (1+2\mathfrak{r}_x) \right] +  \lambda r (1-2\mathfrak{r}_x)  }{2x (x+\lambda)}
\end{align*}
do not depend on $s$. In particular, note that since 
$$\lim_{x \to 0^+} x\cdot A(x),   \lim_{x \to 0^+} x\cdot B(x) \in (0, +\infty) $$
and 
$$\lim_{x \to 0^+} \frac{m_{r; \alpha, \lambda}(x)}{x^r} = \frac{(1-\alpha)_{ \uparrow (r-1)}}{r!} \cdot \lambda^{1-r}, $$
with the particular case
$$\lim_{x \to 0^+} \frac{m_{1; \alpha, \lambda}(x)}{x} = 0, $$
then for every $r, s \ge 1$, 
$$\lim_{x \to 0^+} S_s(x) = 0. $$
For what concerns \eqref{Bis0},  
\begin{align*}
\mathcal{B}_{r; \alpha, \lambda}(x)  &=\left( m_{r; \alpha, \lambda}(x)  \right)^4  \cdot \left\{ -4 \left[A(x) + B(x) \right] + 6[4A(x) + 2B(x)] \right.\\
& \quad \left.-4 \left[9A(x)+ 3B(x)] + [16A(x) + 4B(x)\right]\right\} \\
& = \left( m_{r; \alpha, \lambda}(x)  \right)^4   \cdot \left[A(x) \cdot (-4+24-36+16)  + B(x) \cdot (-4+12-12+4)\right]\\
& = 0.
\end{align*}

\subsubsection{$r = 0$, $\alpha \in (0,1)$}
Based on the distribution of $K_{n}$ \cite[Equation 3.11]{Pit(06)}, a direct calculation shows that
\begin{displaymath}
\mathbb{E} \left[(K_{h} ^{\{n\}})_{\downarrow s}\right] =\frac{\Gamma(\lambda n /\alpha + s)}{\Gamma(\lambda n /\alpha)}
\sum_{i = 0}^s (-1)^{s-i} \binom{s}{i} \frac{\Gamma(\theta + i\alpha + h)}{\Gamma(\lambda n + i\alpha)} \frac{\Gamma(\lambda n )}{\Gamma(\lambda n +h)} 
\end{displaymath}
so that
$$
\mathbb{E} \left[(K_{\lfloor xn \rfloor} ^{\{n\}})_{\downarrow s}\right] = \frac{\Gamma\left(\frac{\lambda n}{\alpha }+s \right)}{\Gamma\left(\frac{\lambda n}{\alpha} \right)} \sum_{k = 0}^s (-1)^{s-k} \binom{s}{k} \frac{\Gamma((\lambda + x) n +k\alpha - \mathfrak{r}_{x}) \, \Gamma(\lambda n)}{\Gamma(\lambda n +k \alpha) \, \Gamma((\lambda +x)n- \mathfrak{r}_{x} )}
$$
where we are denoting by $\mathfrak{r}_x = xn - \lfloor xn \rfloor$. Since $0 \le \mathfrak{r}_x < 1$ for all $x \in [0, 1]$ and for all $n \in \mathbb{N}$, we can apply the asymptotic expansion for ratios of gamma functions \cite[Equation 1]{TE(51)} to obtain that \eqref{falling_expansion} holds with 
$$
\mathcal{S}_s(x) = \begin{cases} 
 \frac{(1-\alpha) x}{2  (\lambda +x)}\left(\frac{\lambda +x}{\lambda}\right)^{\alpha} & \text{if } s = 1\\
  \frac{s(s-1)\alpha}{\lambda }\, m^s_{\alpha, \lambda}(x) -  \left(\frac{\lambda}{\alpha}\right)^2 \, \frac{\alpha x s}{2\lambda (\lambda +x)} \left(\frac{\lambda +x}{\lambda }\right)^\alpha \times&  \\
\quad   \times \left[\left(\frac{\lambda +x}{\lambda }\right)^\alpha (\alpha s -1) + 1 - \alpha \right] m^{s-2}_{\alpha, \lambda}(x)& \text{if } s \ge 2. 
\end{cases}
$$
It remains to show \eqref{Bis0}: 
\begin{align*}
\mathcal{B}_{0;\alpha, \lambda}(x)  & = -4 \, m^3_{\alpha, \lambda}(x) S_1(x)  + 6\, m^2_{\alpha, \lambda}(x) \, S_2(x) -4 m_{\alpha, \lambda}(x)\, S_3(x) + S_4(x)\\
& = m^3_{\alpha, \lambda}(x) \cdot \left[- \frac{2(\alpha -1)}{\lambda +x} \left(\frac{\lambda + x}{\lambda}\right)^\alpha \right] \\
& \quad + m^2_{\alpha, \lambda}(x) \cdot \frac{x\lambda}{x+\lambda} \, \left[  \left(\frac{\lambda + x}{\lambda}\right)^{2\alpha} \left(6 - \frac{2(4\alpha -1)}{\alpha} \right)- \left(\frac{\lambda + x}{\lambda}\right)^{\alpha} \, \frac{2(1-\alpha)}{\alpha} \right]\\
& = 0,
\end{align*}
where the last equality can be easily verified by substituting the expression for $m_{\alpha, \lambda}(x) $.

\subsubsection{$r = 0$, $\alpha = 0$}
\begin{paragraph}
\noindent Arguing as in \cite[Appendix C.2]{Con(25a)}, we obtain
\begin{align*}
\mathbb{E} \left[(K_{\lfloor xn \rfloor} ^{\{n\}})_{\downarrow 1}\right] & = n \cdot \lambda\,  \Phi_0(n, \lambda, x)\\
\mathbb{E} \left[(K_{\lfloor xn \rfloor} ^{\{n\}})_{\downarrow 2}\right] & = n^2 \cdot \lambda^2 \, \left[  \Phi_0^2(n, \lambda, x) +  \Phi_1(n, \lambda, x )\right] \\
\mathbb{E} \left[(K_{\lfloor xn \rfloor} ^{\{n\}})_{\downarrow 3}\right] & = n^3 \cdot \lambda^3 \, \left[\Phi_0^3(n, \lambda, x) + 3 \Phi_0(n, \lambda, x) \Phi_1(n, \lambda, x) + \Phi_2(n, \lambda, x) \right] \\       
\mathbb{E} \left[(K_{\lfloor xn \rfloor} ^{\{n\}})_{\downarrow 4}\right] & = n^4 \cdot \lambda^4 \, \left[\Phi_0^4(n, \lambda, x) + 6 \Phi_0^2(n, \lambda, x) \Phi_1(n, \lambda, x) + 4 \Phi_0(n, \lambda, x) \Phi_2(n, \lambda, x) \right.\\
      & \quad \quad \quad \quad \left.+ 3 \Phi_1^2(n, \lambda, x) + \Phi_3(n, \lambda, x) \right]
 \end{align*}
where
\begin{displaymath}
\Phi_i(n, \lambda, x) := \psi^{(i)} \left(n (\lambda+x)\right)  -\psi^{(i)}\left(n \lambda\right)
\end{displaymath}
and $\psi^{(i)}$ denotes the polygamma function \cite[Section \href{http://dlmf.nist.gov/5.15} {(5.15)}]{nist}. Making use of the asymptotic expansions of the polygamma functions for large argument 
\cite[Equation \href{http://dlmf.nist.gov/5.11.E2} {(5.11.2)}]{nist}, we obtain that, uniformly for $x \in [0,1]$,
\begin{align*}
\Phi_0(n, \lambda, x) & = L(x) + \frac{1}{n} \cdot \frac{x}{2\lambda (x+\lambda)} + O\left(n^{-2}\right)\\
\Phi_1(n, \lambda, x) & = - \frac{1}{n}\cdot  \frac{x}{\lambda (x+\lambda)} + O\left(n^{-2}\right)\\
\Phi_2(n, \lambda, x) & =  O\left(n^{-2}\right)\\
\Phi_3(n, \lambda, x) & =  O\left(n^{-3}\right)\\
\end{align*}
where $L(x)  = \log \left(\frac{\lambda+x}{\lambda} \right)$
There follows that equation \eqref{falling_expansion} holds with $$
\mathcal{S}_s(x) = \begin{cases} 
  \frac{x}{2(\lambda+x)} & \text{if } s = 1\\
\frac{x\lambda}{\lambda+x} [L(x)-1]  & \text{if } s = 2\\
3 \frac{x \lambda^2}{\lambda+x} \left[\frac{1}{2} L^2(x)  -L(x)\right] & \text{if } s = 3\\
2\frac{x \lambda^3}{\lambda+x} [L^3(x) - 3L^2(x)] & \text{if } s = 4
\end{cases}.
$$
Whence,
      \begin{align*}
 \mathcal{B}_{0; 0, \lambda}(x) & = -4 \, \lambda^3 L^3(x) \, \mathcal{S}_1(x)    + 6\, \lambda^2 L^2(x)  \mathcal{S}_2(x)  -4 \lambda L(x)\, \mathcal{S}_3(x) + \mathcal{S}_4(x).\\
 & = \frac{ x \lambda^3}{\lambda +x}\cdot  \left[ L^3(x) \cdot  \left( -2 + 6 - 6+2\right)  + L^2(x) \cdot (-6+12-6) \right]\\
       &= 0.
      \end{align*}
\end{paragraph}

\subsection{Details on the proof of Lemma \ref{mart_lem}} \label{sec_martingale}
It follows from definition \eqref{an_r} that for every $d, n, h$, 
$$
\mathcal{A}_{d, h-1}^{\{n\}} = \operatorname{Diag} \left( \left(\gamma_{r, h}^{\{n\}}\right)_{r \in \{0, ..., d\}}\right) \cdot \mathcal{A}_{d, h}^{\{n\}}
$$
and
$$
\mathbf{A}_{d, h}^{\{n\}} = \mathbf{A}_{d,  h-1}^{\{n\}} + \mathcal{A}_{d, h}^{\{n\}} \cdot  \left(\beta_{r, h-1}^{\{n\}}\right)_{r \in \{0, ..., d\}}
$$
Hence,
\begin{align}\label{deltaform_d_star}
\nonumber \mathbf{\Delta}_{d, h}^{\{n\}} &= \mathcal{A}_{d, h}^{\{n\}} \cdot \mathbf{K}_{d,h}^{\{n\}}  -  \mathbf{A}_{d, h}^{\{n\}} - \mathcal{A}_{d, h-1}^{\{n\}} \cdot \mathbf{K}_{d,h-1}^{\{n\}}  + \mathbf{A}_{d, h-1}^{\{n\}} \\
\nonumber & = \mathcal{A}_{d, h}^{\{n\}}  \cdot \left[ \mathbf{K}_{d,h-1}^{\{n\}} + \mathbf{\Xi}_{d, h}^{\{n\}} \right] -  \mathbf{A}_{d, h-1}^{\{n\}}  - \mathcal{A}_{d, h}^{\{n\}} \cdot  \left(\beta_{r, h-1}^{\{n\}}\right)_{r \in \{0, ..., d\}} \\
\nonumber & \quad -    \operatorname{Diag} \left( \left(\gamma_{r, h}^{\{n\}}\right)_{r \in \{0, ..., d\}}\right) \cdot \mathcal{A}_{d, h}^{\{n\}} \cdot \mathbf{K}_{d,h-1}^{\{n\}}  + \mathbf{A}_{d, h-1}^{\{n\}} \\
\nonumber &= \mathcal{A}_{d, h}^{\{n\}}  \cdot  \left\{ \left[ \mathbf{I}_{d+1}-    \operatorname{Diag} \left( \left(\gamma_{r, h}^{\{n\}}\right)_{r \in \{0, ..., d\}}\right)\right] \cdot  \mathbf{K}_{d,h-1}^{\{n\}} 
+ \, \mathbf{\Xi}_{d, h}^{\{n\}} -   \left(\beta_{r, h-1}^{\{n\}}\right)_{r \in \{0, ..., d\}}  \right\}\\
& = \mathcal{A}_{d, h}^{\{n\}}  \cdot  \left[ \mathbf{\Xi}_{d, h}^{\{n\} }- \left(p_{r, h-1}^{\{n\}} - q_{r, h-1}^{\{n\}}\right)_{r \in \{0, ..., d\}} \right], 
\end{align}
where the last equality follows from \eqref{xi_mean_r}. Taking conditional expectations on both side we obtain
$$
\mathbb{E} \left[\mathbf{\Delta}_h^{\{n\}}  \, | \, \mathcal{F}_{h-1}^{\{n\}} \right] = \mathcal{A}_h^{\{n\}} \cdot \mathbb{E} \left[ \mathbf{\Xi}_{d, h}^{\{n\} }\, | \, \mathcal{F}_{h-1}^{\{n\}} \right]  -  \left(p_{r, h-1}^{\{n\}} - q_{r, h-1}^{\{n\}}\right)_{r \in \{0, ..., d\}}  =\mathbf{0},
$$
which concludes the proof.

\subsection{Details on the proof of Lemma \ref{increasingproc_lem_d}}\label{app3_2}
We begin by proving the following lemma:
\begin{lem}\label{as_1_d}
As $n \to +\infty$, for all $i  \in \{1, ..., d+1\}$, 
\begin{enumerate}[(i)]
\item 
$$\lim_{n \to +\infty} a_{i-1, \lfloor n x\rfloor  + 1}^{\{n\}} = \left(\frac{\lambda + x}{\lambda}\right)^{i-1-\alpha}$$
\item
$$
p_{i-1, \lfloor nx \rfloor}^{\{n\}} \stackrel{a.s.}{\longrightarrow} \begin{cases}
\left(\frac{\lambda + x}{\lambda}\right)^{\alpha -1 } & \text {if } i=1\\
\frac{(1-\alpha)_{(i-2) \uparrow} \,  \lambda^{1-\alpha}}{(i-2)!}  \cdot x^{i-2} (x+\lambda)^{\alpha -i-1}& \text {if } i\ge 2\\
 \end{cases}
$$
\item
$$
q_{i-1, \lfloor nx \rfloor}^{\{n\}} \stackrel{a.s.}{\longrightarrow} \begin{cases}
0& \text {if } i=1\\
\frac{(1-\alpha)_{(i-1) \uparrow} \,  \lambda^{1-\alpha}}{(i-1)!}  \cdot x^{i-1} (x+\lambda)^{\alpha -i}& \text {if } i\ge 2\\
 \end{cases}
$$
\end{enumerate}
where convergence in (i), (ii) and (ii) is uniform for $x \in [0,1]$.
\end{lem}
\begin{proof}
By definition of $a_{i-1, \lfloor n x\rfloor  + 1}^{\{n\}} $ \eqref{an_r}, 
\begin{align*}
a_{i-1, \lfloor n x\rfloor  + 1}^{\{n\}} & = \frac{(\lambda n + 1)_{\lfloor nx \rfloor \uparrow} }{(\lambda n -i + 2 + \alpha)_{\lfloor nx \rfloor \uparrow }}\\
& = \frac{\Gamma((\lambda +x) n  + 1 - \mathfrak{r}_x)\cdot \Gamma(\lambda n - i+2+\alpha)}{\Gamma((\lambda +x) n  -i+2+\alpha - \mathfrak{r}_x)\cdot \Gamma(\lambda n +1 )}
\end{align*}
where we are denoting by $\mathfrak{r}_x = xn - \lfloor xn \rfloor$. Since $0 \le \mathfrak{r}_x < 1$ for all $x \in [0, 1]$ and for all $n \in \mathbb{N}$, we can apply the asymptotic expansion for ratios of gamma functions \cite[Equation 1]{TE(51)} to obtain the desired asymptotics.

\begin{paragraph}
\noindent For what concerns $p_{r, \lfloor nx \rfloor}^{\{n\}}$, recall 
\begin{align*}
 p_{r, h-1}^{\{n\}} & = \begin{cases}
\frac{\alpha K_{\lfloor nx \rfloor }^{\{n\}} + \lambda n}{\lambda n + \lfloor nx \rfloor} & \text{ if } r =0, 1\\
\frac{(r-1-\alpha) K_{r-1, \lfloor nx \rfloor }^{\{n\}}}{\lambda n + \lfloor nx \rfloor } & \text{ if } r \ge 2
\end{cases}\\
& =  \begin{cases}
\frac{\alpha}{\left(x+ \lambda - \frac{\mathfrak{r}_x}{n} \right) } \cdot \frac{K_{\lfloor nx \rfloor }^{\{n\}}}{n}  + \frac{\lambda}{\left(x+ \lambda - \frac{\mathfrak{r}_x}{n} \right)} & \text{ if } r =0, 1\\
\frac{(r-1-\alpha)}{\left(x+ \lambda - \frac{\mathfrak{r}_x}{n} \right)} \cdot \frac{ K_{r-1, \lfloor nx \rfloor }^{\{n\}}}{n } & \text{ if } r \ge 2
\end{cases}
\end{align*}
since $\sup_{x} \mathfrak{r}_x / n < 1/n \to 0$, it sufficies to apply lemma \ref{Kxrnn} to obtain the desired asymptotics.
\end{paragraph}

\begin{paragraph}
\noindent Analogously, 
$$ q_{r, h-1}^{\{n\}} = \begin{cases}
0 & \text{ if } r =0\\
\frac{(r-\alpha) K_{r, \lfloor nx \rfloor}^{\{n\}}}{\lambda n +\lfloor nx \rfloor } & \text{ if } r \ge 1
\end{cases}
= \begin{cases}
0 & \text{ if } r =0\\
\frac{(r-\alpha)}{\left(x+ \lambda - \frac{\mathfrak{r}_x}{n} \right)} \cdot  \frac{ K_{r, \lfloor nx \rfloor}^{\{n\}}}{n}& \text{ if } r \ge 1
\end{cases}
$$
and the desired asymptotics follows from lemma \ref{Kxrnn}.
\end{paragraph}
\end{proof}

Recalling the definitions of $\left(P_{ \lfloor n x\rfloor }^{\{n\}}\right)_{i, j} $ and $ \left(R_{ \lfloor n x\rfloor }^{\{n\}}\right)_{i, j}$, lemma \ref{as_1_d} implies

$$
\left(P_{\lfloor nx\rfloor}^{\{n\}}\right)_{i,j}
\stackrel{a.s.}{\longrightarrow}
\begin{cases}
0
& \text{if } j \ge i+2 \\[6pt]

-\,c_{i-1}\,\lambda^{1-\alpha}\,x^{\,i-1}(x+\lambda)^{\,\alpha-i} 
& \text{if } j=i+1, \ i\ge2 \\[8pt]

\left(\frac{x+\lambda}{\lambda}\right)^{\alpha-1}
& \text{if } i=j=1 \text{ or } i=1,j=2 \\[6pt]

c_{i-2}\,\lambda^{1-\alpha}\,x^{\,i-2}(x+\lambda)^{\,\alpha-i-1} & \\
\quad +
c_{i-1}\,\lambda^{1-\alpha}\,x^{\,i-1}(x+\lambda)^{\,\alpha-i}
& \text{if } i=j\ge2.
\end{cases}
$$
and 
$$
\left(R_{\lfloor nx\rfloor}^{\{n\}}\right)_{i,j}
\stackrel{a.s.}{\longrightarrow}
\begin{cases}
\left(\frac{x+\lambda}{\lambda}\right)^{2(\alpha-1)}
& \text{if } i=j=1, \\[8pt]

\lambda^{2(1-\alpha)} 
\Big[
c_{j-2}\,x^{\,j-2}(x+\lambda)^{\,2\alpha-j-2}
-
c_{j-1}\,x^{\,j-1}(x+\lambda)^{\,2\alpha-j-1}
\Big]
& \text{if } i=1,\ j\ge2 \\[6pt]

\lambda^{2(1-\alpha)}\Big\{
c_{i-2}\,c_{j-2}\,x^{\,i+j-4}(x+\lambda)^{\,2\alpha-i-j-2} \\[2pt]
\quad
- \big[ c_{i-1}\,c_{j-2} + c_{i-2}\,c_{j-1} \big]\,x^{\,i+j-3}(x+\lambda)^{\,2\alpha-i-j-1} \\[1pt]
\quad
+ c_{i-1}\,c_{j-1}\,x^{\,i+j-2}(x+\lambda)^{\,2\alpha-i-j} \Big\}
& \text{if } i,j\ge2.
\end{cases}
$$
uniformly for $x \in [0,1]$, where 
$$
c_r = \frac{(1-\alpha)_{r \uparrow} }{r!}
$$
Finally, recalling that
\begin{equation*}
F^{\{n\}}_{i, j}(x)  = a_{i-1, \lfloor n x\rfloor  + 1}^{\{n\}} a_{j-1, \lfloor n x\rfloor  + 1}^{\{n\}} \cdot \left[ \left(P_{ \lfloor n x\rfloor }^{\{n\}}\right)_{i, j}  - \left(R_{ \lfloor n x\rfloor }^{\{n\}}\right)_{i, j} \right]
\end{equation*}
all the above results entail $F^{\{n\}}_{i,j}(x) \stackrel{a.s.}{\longrightarrow} f_{i,j}(x)$ uniformly for $x \in [0,1]$, where 
\begin{equation}\label{fij}
\small
f_{i,j} (x) = \begin{cases}
 \left(\frac{\lambda+x}{\lambda}\right)^{-1-\alpha} -  \left(\frac{\lambda+x}{\lambda}\right)^{-2} & \text{ if } i = j = 1\\[10pt]
 \left(\frac{\lambda+x}{\lambda}\right)^{-\alpha} - \lambda \cdot \left(x+\lambda\right)^{-3} + \lambda \, c_1 \cdot x \, (x+\lambda)^{-2}& \text{ if } i = 1, j = 2\\[10pt]
\lambda^{3-j} \cdot \left[ - c_{j-2} \cdot x^{j-2} \left(x+\lambda\right)^{-3} +  c_{j-1} \cdot x^{j-1} \, (x+\lambda)^{-2} \right]& \text{ if } i = 1, j \ge 3\\[10pt]
\lambda^{\alpha + 3 -2i} \left[c_{i-2}\cdot x^{i-2}(x+\lambda)^{i-3-\alpha} + c_{i-1} \cdot x^{i-1}(x+\lambda)^{i-2-\alpha} \right]&\\[10pt]
 \quad - \lambda^{4-2i}\Big\{
c_{i-2}^2 \cdot x^{2i-4}(x+\lambda)^{-4} \\[1pt]
\quad \qquad \qquad
- 2 c_{i-1}\,c_{i-2} \cdot x^{2i-3}(x+\lambda)^{-3} \\[1pt]
\quad \qquad \qquad 
+ c_{i-1}^2 \cdot \,x^{2i-2}(x+\lambda)^{-2} \Big\} & \text{ if }i = j \ge 2\\[10pt]
- c_{i-1} \lambda^{2+\alpha-2i} \cdot x^{i-1} (x+\lambda)^{i-1-\alpha}\\
\quad - \lambda^{3-2i}\Big\{
c_{i-2}\,c_{i-1}\cdot x^{2i-3}(x+\lambda)^{-4} \\[1pt]
\qquad \qquad \quad
- \big[ c_{i-1}\,c_{i-1} + c_{i-2}\,c_{i} \big]\cdot x^{2i-2}(x+\lambda)^{-3} \\[1pt]
\qquad \qquad \quad
+ c_{i-1}\,c_{i}\,x^{2i-1}(x+\lambda)^{-2} \Big\}& \text{ if } i \ge 2, j = i+1\\[10pt]
- \lambda^{4-i-j} \cdot \Big\{
c_{i-2}\,c_{j-2}\cdot x^{\,i+j-4}(x+\lambda)^{-4} \\[1pt]
\quad\qquad \qquad
- \big[ c_{i-1}\,c_{j-2} + c_{i-2}\,c_{j-1} \big]\cdot x^{\,i+j-3}(x+\lambda)^{-3} \\[1pt]
\quad\qquad\qquad
+ c_{i-1}\,c_{j-1}\cdot x^{\,i+j-2}(x+\lambda)^{- 2} \Big\}
 & \text{ if } i \ge 2, j \ge i+2\\[10pt]
 f_{j,i} (x)  & \text{ if } i>j
\end{cases}
\end{equation}

All that remains to prove is that 
$$
\mathcal{I}_{i,j} : = \int_{0}^1 f_{i,j} (x) \, \mathrm{d}x = \left(\Gamma_{\alpha, \lambda}\right)_{i,j}.
$$
Exploiting the identities \cite[Equation 3.197.8]{GrRy(07)} and \cite[Equation 15.3.3]{AbSt(64)} we obtain
\begin{align*}
\int_{0}^{1} x^a (x+\lambda)^b  \mathrm{d}x&= \frac{\lambda^{b}}{a+1} \cdot {}_2F_1 \left(a+1, -b ; a+2 ; -\frac{1}{\lambda}\right)\\
& =
\frac{(\lambda+1)^{b}}{\lambda(a+1)} \cdot {}_2F_1 \left(1, a+b+2 ; a+2 ; -\frac{1}{\lambda}\right) \\
& = \frac{(\lambda+1)^{b}}{\lambda(a+1)} \cdot H(a+b+2, a+2)
\end{align*}
Applying such identity allows, after straightforward but tedious computations, to show that for every  $i, j \in \{0, ..., d\}$

\begin{equation}\label{int_clean}
\small
\mathcal{I}_{i,j} (x) = \begin{cases}
s^2_{\alpha, \lambda} & i = j = 1\\[15pt]
\frac{\lambda}{\alpha -1}  - \frac{(\lambda + 1)}{\alpha -1} \cdot  \left( \frac{\lambda}{\lambda+1} \right)^{\alpha}- \left(\frac{\lambda}{\lambda+1}\right)^2+ \frac{(1-\alpha)  \lambda }{2\,  (\lambda+1)}\cdot \log\left(\frac{\lambda+1}{\lambda}\right) & i = 1, j = 2\\[15pt]
\lambda^{2-j} \Big[ 
- c_{j-2} \frac{(\lambda+1)^{-3}}{j-1} \, _2F_1(1, j-3; j-1; -1/\lambda) &\\
\quad \qquad+ c_{j-1} \frac{(\lambda+1)^{-2}}{j} \, _2F_1(1, j-1; j; -1/\lambda)
\Big] & i = 1, j \ge 3\\[15pt]
\lambda^{\alpha + 2 -2i} \Big[
c_{i-2} \frac{(\lambda+1)^{i-3-\alpha}}{i-1} \, _2F_1(1, i-\alpha-2; i-1; -1/\lambda) &\\
\quad \qquad \quad + c_{i-1} \frac{(\lambda+1)^{i-2-\alpha}}{i} \, _2F_1(1, i-\alpha; i; -1/\lambda)
\Big] &\\
\quad - \lambda^{3-2i} \Big[
c_{i-2}^2 \frac{(\lambda+1)^{-4}}{2i-3} \, _2F_1(1, 2i-2; 2i-3; -1/\lambda)& \\
\quad \qquad\quad - 2 c_{i-1} c_{i-2} \frac{(\lambda+1)^{-3}}{2i-2} \, _2F_1(1, 2i-1; 2i-2; -1/\lambda) & \\
\quad \qquad\quad
+ c_{i-1}^2 \frac{(\lambda+1)^{-2}}{2i-1} \, _2F_1(1, 2i; 2i-1; -1/\lambda)
\Big] & i = j \ge 2\\[15pt]
- c_{i-1} \lambda^{1+\alpha-2i} \frac{(\lambda+1)^{i-1-\alpha}}{i} \, _2F_1(1, i-\alpha; i; -1/\lambda) &\\
\quad - \lambda^{2-2i} \Big[
c_{i-2} c_{i-1} \frac{(\lambda+1)^{-4}}{2i-2} \, _2F_1(1, 2i-1; 2i-2; -1/\lambda)& \\
\quad \qquad\quad
- (c_{i-1}^2 + c_{i-2} c_i) \frac{(\lambda+1)^{-3}}{2i-1} \, _2F_1(1, 2i; 2i-1; -1/\lambda) \\
\quad \qquad\quad + c_{i-1} c_i \frac{(\lambda+1)^{-2}}{2i} \, _2F_1(1, 2i+1; 2i; -1/\lambda)
\Big] & i \ge 2, j = i+1\\[15pt]
- \lambda^{3-i-j} \Big[
c_{i-2} c_{j-2} \frac{(\lambda+1)^{-4}}{i+j-3} \, _2F_1(1, i+j-2; i+j-3; -1/\lambda) \\[1pt]
\quad \qquad\quad- (c_{i-1} c_{j-2} + c_{i-2} c_{j-1}) \frac{(\lambda+1)^{-3}}{i+j-2} \, _2F_1(1, i+j-1; i+j-2; -1/\lambda) \\[1pt]
\quad \qquad\quad + c_{i-1} c_{j-1} \frac{(\lambda+1)^{-2}}{i+j-1} \, _2F_1(1, i+j; i+j-1; -1/\lambda)
\Big] & i \ge 2, j \ge i+2\\[15pt]
f_{j,i} (x) & i>j
\end{cases}
\end{equation}
which matches the definition of $\Gamma_{\alpha, \lambda}$.

\end{document}